\documentclass{amsart}
\usepackage[margin=1.5cm]{geometry}
\numberwithin{equation}{section}
\usepackage{amssymb}
\usepackage{amsmath, amsfonts,amsthm,amssymb,amscd, hyperref, verbatim,graphicx,color,multirow,booktabs, caption,tikz,tikz-cd, mathdots,bm}
\usepackage{tikz-cd}
\usetikzlibrary{positioning}
\newtheorem{theorem}{Theorem}
\newtheorem{lemma}{Lemma}[section]

\newtheorem{remark}[lemma]{Remark}

\newcommand{\Sym}{\mathop{\mathrm{Sym}}}
\def\nor#1#2{{\bf N}_{{#1}}({{#2}})}
\hypersetup{
colorlinks=false,
allbordercolors=red,
citebordercolor=green
}

\begin{document}
\title[Abelian quotients of transitive groups]{A subexponential bound on the cardinality of abelian quotients\\in finite transitive groups}
\author[A. Lucchini]{Andrea Lucchini}
\address{Andrea Lucchini, Dipartimento di Matematica Pura e Applicata,\newline
 University of Padova, Via Trieste 53, 35121 Padova, Italy} 
\email{lucchini@math.unipd.it}
         
\author[L. Sabatini]{Luca Sabatini}
\address{Luca Sabatini, Dipartimento di Matematica  e Informatica ``Ulisse Dini'',\newline
 University of Firenze, Viale Morgagni 67/a, 50134 Firenze, Italy} 
\email{luca.sabatini@unifi.it}
\author[P. Spiga]{Pablo Spiga}
\address{Pablo Spiga, Dipartimento di Matematica Pura e Applicata,\newline
 University of Milano-Bicocca, Via Cozzi 55, 20126 Milano, Italy} 
\email{pablo.spiga@unimib.it}
\subjclass[2010]{primary 20B30, 20B35}
\keywords{permutation groups; symmetric group; abelian quotients}        
	
\maketitle

        \begin{abstract}
        We show that, for every transitive permutation group $G$ of degree $n\ge 2$,
        the largest abelian quotient of $G$ has cardinality at most $4^{n/\sqrt{\log_2 n}}$.
        This gives a positive answer to a 1989 outstanding question of L\'aszl\'o Kov\'acs and Cheryl Praeger. 
       \end{abstract}
          
\section{Introduction}

L\'aszl\'o Kov\'acs and Cheryl Praeger~\cite{KP} have investigated large abelian quotients in arbitrary permutation groups of finite degree.
Their work was motivated by recent (at that time) investigations on minimal permutation representations of a finite group~\cite{EP}.
One of the main results in~\cite{KP} (which is independently proved in \cite{AG}) shows that,
for every permutation group of degree $n$, the largest abelian quotient has order at most $3^{n/3}$. Clearly, this bound is attained, whenever $n$ is a multiple of $3$,
by an elementary abelian $3$-group of order $3^{n/3}$ having all of its orbits of cardinality $3$.
Furthermore, the authors conjecture that, for \textit{transitive} groups of degree $n$, a subexponential bound in $n(\log_2 n)^{-1/2}$ holds.
More history on this conjecture and more details can be found in the survey paper~\cite{PS}.

The first substantial evidence towards the conjecture goes back to the work of Aschbacher and Guralnick~\cite{AG};
they proved the striking result that the largest abelian quotient of a \textit{primitive} group of degree $n$ has order at most $n$.
In the concluding remarks, the authors also independently ask whether
one can obtain a subexponetial bound on the order of abelian quotients of transitive groups in terms of their degrees.
We refer to~\cite{AG,PS} for an infinite family of transitive groups $G$ of degree $n$ with $|G/G'|$ asymptotic to $\exp(bn/\sqrt{\log_2 n})$,
for some constant $b$.

The second substantial evidence towards the conjecture is in~\cite{GMP},
where many of the results in Section~7 get very close to the desired upper bound.
In particular, Theorem 7.6 in \cite{GMP} says that if $G$ is a transitive permutation group of degree $n \ge 2$ and 
 $N \lhd G$ is a still transitive normal subgroup of $G,$ then
the product of the orders of the abelian composition factors of $G/N$ is at most $4^{n/\sqrt{\log_2 n}}$.

In this paper, we settle in the affirmative the conjecture of Kov\'acs and Praeger.
\begin{theorem}\label{thrm:main}For every positive integer $n\ge 2$ and for every transitive permutation group $G$ of degree $n$, we have
$$ |G/G'| \> \le \> 4^{n/\sqrt{\log_2 n}} . $$
\end{theorem}

The constant $4$ in Theorem~\ref{thrm:main} should not be taken too seriously, but it seems remarkably hard to pin down the exact constant.
The choice of the constant $4$ in our work is a compromise:
it makes the statement of Theorem~\ref{thrm:main} explicit and valid for every $n\ge 2$. \\

\section{Preliminaries}\label{preliminaries}
Unless otherwise explicitly stated, all the logarithms are to base $2$.
Given a field $\mathbb{F}$, a  group $G$, a subgroup $H$ of $G$ and an $\mathbb{F}H$-module $W$ (or simply $H$-module), we denote by $W\uparrow_{H}^G$ the induced $G$-module of $W$ from $H$ to $G$, that is, $W\uparrow_H^G:=W\otimes_{\mathbb{F}H}\mathbb{F}G$.  Moreover, given a $G$-module $M$, we denote by $d_G(M)$ the minimal number of generators of $M$ as a $G$-module.
We are ready to report a fundamental result from~\cite{LMM}.

\begin{lemma}\emph{{(See~\cite[Lemma~4]{LMM})}}\label{l:LEMMA}
There is a universal constant $b'$ such that whenever $H$ is a subgroup of index $n\ge 2$ in a finite group $G$,
$\mathbb{F}$ is a field, $V$ is an $H$-module of dimension $a$ over $\mathbb{F}$
and $M$ is a $G$-submodule of the induced module $V\uparrow_H^G$, then
$$d_G(M) \> \le \> \frac{ab'n}{\sqrt{\log n}} . $$ 
\end{lemma}

\begin{remark}\label{remark}
{\rm Gareth Tracey, in his monumental work~\cite{Tracey} on minimal sets of generators of transitive groups, has refined Lemma~\ref{l:LEMMA} in various directions. For instance,~\cite[Section~4.3]{Tracey} gives a more quantitative form of Lemma~\ref{l:LEMMA}. Indeed, using the notation in Lemma~\ref{l:LEMMA}, from~\cite[Corollary~$4.27$~(iii)]{Tracey}, we deduce 
\[
d_G(M) \le aE(n,p)\le \begin{cases}
an\frac{2}{c'\log n}&\textrm{when }2\le n\le 1260,\\
an\frac{2}{\sqrt{\pi\log n}}&\textrm{when }n>1261,
\end{cases}
\]
where $c':=0.552282$, $p$ is the characteristic of $M$ and $E(n,p)$ is explicitly defined in~\cite[Section~4]{Tracey}. In particular, we immediately see that in Lemma~\ref{l:LEMMA} we may take $b':=2/\sqrt{\pi}$ whenever $n>1261$. With the help of a computer, we have implemented the function $E(n,p)$ and we have checked that $E(n,p)\le 2n/\sqrt{\pi\log n}$ also when $n\le 1260$. Therefore in Lemma~\ref{l:LEMMA} we may take $b':=2/\sqrt{\pi}$.
}
\end{remark}

Let $R$ be a finite group.
For each prime number $p$, let $a_p(R)$ be the number of abelian composition factors of $R$ of order $p$,
and let 
$$ a(R) \> := \> \sum_{p\,\mathrm{prime}}a_p(R)\log p . $$
We now report a useful result of Pyber.

\begin{lemma}\emph{{(See~\cite[Theorem~$2.10$]{Pyber})}}\label{l:LEMMA1}
Let  $c_0:=\log_9(48\cdot 24^{1/3})$.
The product of the orders of the abelian composition factors of a primitive permutation group of degree $r$ is at most $24^{-1/3}r^{1+c_0}$.
\end{lemma}

From Lemma~\ref{l:LEMMA1}, we deduce the following.

\begin{lemma}\label{l:LEMMA2}
Let $R$ be a primitive group of degree $r$, let $c_0$ be the constant in Lemma~$\ref{l:LEMMA1}$.
Then
$$ a(R) \> \le \> (1+c_0)\log r-\log(24)/3. $$
\end{lemma}
\begin{proof}
By definition, the product of the orders of the abelian composition factors of $R$ is $$\prod_{p\,\mathrm{ prime}}p^{a_p(R)}=\prod_{p\,\mathrm{prime}}2^{a_p(R)\log p}=2^{a(R)}.$$ From Lemma~\ref{l:LEMMA1}, this number is at most $24^{-1/3}r^{1+c_0}$.
The proof follows by taking logarithms.
\end{proof}

Notice that Lemma \ref{l:LEMMA1} is often used in order to bound the composition length of a primitive permutation groups. 
A more precise bound on this composition length has been recently proved by  Glasby, Praeger, Rosa and Verret \cite[Theorem 1.3]{gprv}. However this stronger bound is not sufficient for our application, which requires information not only on the number of the composition factors but also on their order.

Finally, given a finite group $G$, we denote by $G_{\mathrm{ab}}$ the quotient group $G/G'$.\\

\section{Proof of Theorem~$\ref{thrm:main}$}

Let $R$ be a finite group, let $\Delta$ be a finite set and let  $W:=R\,\mathrm{wr}_\Delta\Sym(\Delta)$ be the wreath product of $R$ via $\Sym(\Delta)$. We denote by $$\pi:W\to \Sym(\Delta)$$
the projection of $W$ over the top group $\Sym(\Delta)$. Let $\prod_{\delta\in \Delta}R_\delta$ be the base subgroup of $W$ and, for each $\delta\in \Delta$, consider $W_\delta:=\nor W{R_\delta}$. As $$W_\delta=R_\delta\times R\,\mathrm{wr}\Sym(\Delta\setminus\{\delta\}),$$
we may consider the projection $\rho_\delta:W_\delta\to R_\delta$.
Using this notation, we adapt the proof of~\cite[Lemma~2.5]{L} to prove the following.

\begin{lemma}\label{L:lemma}
Let $R$ be a  finite group, let $\Delta$ be a set of cardinality at least $2$
and let $G$ be a subgroup of the wreath product $R\,\mathrm{wr}_\Delta\Sym(\Delta)$ with the properties
\begin{enumerate}
\item\label{enu0}$\pi(G)$ is transitive on $\Delta$,
\item\label{enu1}$\rho_\delta(\nor G{R_\delta})=R_\delta$, for every $\delta\in \Delta.$
\end{enumerate}
Then 
$$ \log|G_{\mathrm{ab}}| \> \le \> \frac{a(R)b'|\Delta|}{\sqrt{\log|\Delta|}}+\log|(\pi(G))_{\mathrm{ab}}|, $$
where $b'$ is the absolute constant appearing in Lemma~$\ref{l:LEMMA}$, and $a(R)$ is defined in Section~$\ref{preliminaries}$.
\end{lemma}
\begin{proof}
We argue by induction on the order of $R$. When $|R|=1$, there is nothing to prove because $\pi(G)\cong G$ and hence $\log|G_{\mathrm{ab}}|=\log|(\pi(G))_{\mathrm{ab}}|$. Suppose then $R\ne 1$.
We write
\begin{align}\label{eq:22}
|G_{\mathrm{ab}}|=|G:G'M||G'M:G'|=|(G/M)_{\mathrm{ab}}||M:M\cap G'| .
\end{align}

 Let $L$ be a minimal normal subgroup of $R$. Fix $\delta_0\in \Delta$. We identify $L$ with a normal subgroup $L_{\delta_0}$ of the direct factor $R_{\delta_0}$ of the base group $\prod_{\delta\in \Delta}R_\delta$ of $W$. Let $B_L$ be the direct product of the distinct $G$-conjugates of $L_{\delta_0}$ and consider $M:=B_L\cap G$. We have $M\unlhd G$ and
$$\frac{G}{M}=\frac{G}{B_L\cap G}\cong \frac{GB_L}{B_L}.$$
Now, from~\eqref{enu0}, we deduce that $GB_L/B_L$ is isomorphic to a subgroup of the wreath product $$(R/L)\mathrm{wr}_\Delta\Sym(\Delta).$$

Therefore, by induction,
\begin{align}\label{eq:11}
\log|(G/M)_{\mathrm{ab}}|\le 
\frac{a(R/L)b'|\Delta|}{\sqrt{\log|\Delta|}}+
\log|(\pi(G))_{\mathrm{ab}}|.
\end{align}

We now distinguish two cases. 
\smallskip

\noindent\textsc{$L$ is non-abelian:}

\smallskip

\noindent Since $M\unlhd W_{\delta_0}\cap G$, we deduce 
$\rho_{\delta_0}(M)\unlhd \rho_{\delta_0}(W_{\delta_0}\cap G)$. From~\eqref{enu1}, we have $\rho_{\delta_0}(W_{\delta_0}\cap G)=\rho_{\delta_0}(\nor G{R_{\delta_0}})=R_{\delta_0}$ and hence
$\rho_{\delta_0}(M)\unlhd R_{\delta_0}$. Observe that $\rho_{\delta_0}(M)$ is contained in $L_{\delta_0}$. As $L_{\delta_0}$ is a minimal normal subgroup of $R_{\delta_0}$, we get either $\rho_{\delta_0}(M)=1$ or $\rho_{\delta_0}(M)=L_{\delta_0}$.  From~\eqref{enu0}, $\pi(G)$ is transitive on $\Delta$ and hence either $\rho_\delta(M)=1$ for each $\delta\in \Delta$, or $\rho_\delta(M)=L_\delta$ for each $\delta\in \Delta$.

Suppose $\rho_{\delta_{0}}(M)=1$. As $\rho_\delta(M)=1$ for each $\delta\in \Delta$, we get $M=1$. Now the proof immediately follows from~\eqref{eq:11} because $G/M\cong G$.

Suppose $\rho_{\delta_0}(M)=L_{\delta_0}$. Then $M$ is a subdirect product of $L^\Delta=\prod_{\delta\in \Delta}L_\delta$. As $L$ is a non-abelian minimal normal subgroup of $R$, we deduce that $M$ is a direct product of non-abelian simple groups. Thus $M$ has no abelian composition factor and hence~\eqref{eq:22} gives $|G_{\mathrm{ab}}|=|(G/M)_{\mathrm{ab}}|$. Moreover,  $a(R/L)=a(R)$ and hence, once again, the proof immediately follows from~\eqref{eq:11}.

\smallskip

\noindent\textsc{$L$ is abelian:}

\smallskip

\noindent  As $L$ is a minimal normal subgroup of $R$, it is an elementary abelian $p_0$-group, for some prime number $p_0$. Let $a_{p_0}$ be the composition length of $L$. In particular,
$$a(R)=a(R/L)+a_{p_0}\log {p_0}.$$

The group $B_L$ is abelian and the action of $G$ by conjugation on $B_L$ endows $B_L$ with a natural structure of $G$-module.
From its definition, as $G$-module, $B_L$ is isomorphic to the induced module 
$$L_{\delta_0}\uparrow_{K}^G,$$
where $K:=\nor G{L_{\delta_0}}$. From~\eqref{enu0}, $G$ acts transitively on $\Delta$ and hence $|\Delta|=|G:\nor G {L_{\delta_0}}|=|G:K|$. From Lemma~\ref{l:LEMMA}, we deduce 
$$ d_G(M/(M \cap G')) \le d_G(M)\le \frac{a_{p_0}b'|\Delta|}{\sqrt{\log|\Delta|}} . $$ 
However, as $G$ acts trivially by conjugation on $M/(M \cap G')$,
we get that $d_G(M/(M\cap G'))$ is just the dimension of $M/(M\cap G')$ as a vector space over the prime field $\mathbb{Z}/p_0\mathbb{Z}$.
Therefore
\begin{align}\label{eq:33}
|M:M\cap G'|\le p_0^{ ( a_{p_0}b'|\Delta|/\sqrt{\log|\Delta|} ) }.
\end{align}

From~\eqref{eq:22}, \eqref{eq:11}, and~\eqref{eq:33}, we get 
\begin{align*}
\log|G_{\mathrm{ab}}|\le& 
\log|(G/M)_{\mathrm{ab}}| +\log|M:M\cap G'|\\
\leq&\frac{a(R/L)b'|\Delta|}{\sqrt{\log|\Delta|}}
+
\log|(\pi(G))_{\mathrm{ab}}|
+\log(p_0) \frac{a_{p_0}b'|\Delta|}{\sqrt{\log|\Delta|}} \\
=&(a(R/L)+a_{p_0}\log{p_0})\frac{b'|\Delta|}{\sqrt{\log|\Delta|}}+
\log|(\pi(G))_{\mathrm{ab}}|\\
=&a(R)\frac{b'|\Delta|}{\sqrt{\log|\Delta|}}+
\log|(\pi(G))_{\mathrm{ab}}|.\qedhere
\end{align*}
\end{proof} 
\vspace{0.2cm}

 With Lemma~\ref{L:lemma} in hand, we prove Theorem~\ref{thrm:main} by induction on $n$. 
 
Let $G$ be a transitive permutation group of degree $n\ge 2$.
From the main result of~\cite{KP}, we have $|G_{\mathrm{ab}}|\le 3^{n/3}$.
Now the inequality $3^{n/3}\le 4^{n/\sqrt{\log n}}$ is satisfied for each $n\le 20\, 603$.
In particular, for the rest of the proof, we may suppose that $n\ge 20\, 604$.

Suppose first that $G$ is primitive. In this case, from~\cite{AG}, we have $|G_{\mathrm{ab}}|\le n$ and the inequality $n\le 4^{n/\sqrt{\log n}}$ follows with an easy computation.

Suppose now that $G$ is imprimitive and let $\Omega$ be the domain of $G$. Among all non-trivial blocks of imprimitivity of $G$,
choose one (say $\Lambda$) minimal with respect to the inclusion. Let $G_{\{\Lambda\}}:=\{g\in G\mid \Lambda^g=\Lambda\}$ be the setwise stabilizer of $\Lambda$ in $G$ and let $R\le \Sym(\Lambda)$  be the permutation group induced by $G_{\{\Lambda\}}$ in its action on $\Lambda$. The minimality of $\Lambda$ yields that $R$ acts primitively on $\Lambda$.  

Let $\Delta:=\{\Lambda^g\mid g\in G\}$ be the system of imprimitivity determined by the block $\Lambda$. Then $G$ is a subgroup of the wreath product
$$R\,\mathrm{wr}_\Delta\Sym(\Delta).$$
We now use the notation of Lemma~\ref{L:lemma} for wreath products.
In particular, let $\pi:R\,\mathrm{wr}_\Delta \Sym(\Delta)\to \Sym(\Delta)$ be the projection onto the top group $\Sym(\Delta)$ and, for each $\delta\in \Delta$, let $R_\delta$ be the direct factor of the base group $\prod_{\delta\in \Delta}R_\delta$ corresponding to $\delta$.
From the fact that $G$ acts transitively on $\Omega$ and from the definition of $R$, we get that the two hypotheses~\eqref{enu0} and~\eqref{enu1} are satisfied. Therefore, from Lemma~\ref{L:lemma} itself, we deduce
$$\log|G_{\mathrm{ab}}|\le \frac{a(R)b'|\Delta|}{\sqrt{\log|\Delta|}}+
\log|(\pi(G))_{\mathrm{ab}}|.$$

Set $r:=|\Lambda|$. Thus $|\Delta|=n/r$. From Lemma~\ref{l:LEMMA2} and from induction (as $n/r<n$), we get
\begin{equation}\label{aux}
 \log|G_{\mathrm{ab}}| \le 
 \frac{b' (n/r)}{\sqrt{\log (n/r)}} \left( (1+c_0) \log r - \frac{\log(24)}{3} \right) +2\frac{(n/r)}{\sqrt{\log(n/r)}} . 
\end{equation}
From Remark~\ref{remark}, we see that we may take $b'=2/\sqrt{\pi}$. Now, for $n\ge 20\,604$, a careful calculation shows that the right hand side of~\eqref{aux} is at most $2n/\sqrt{\log n}$ for every divisor $r$ of $n$ with $4<r<n$.

We now discuss the cases $r\in \{2,3,4\}$ separately.
When $r=2$, we have $a(R)=1$ and hence
\begin{equation}\label{aux1}
\log|G_{\mathrm{ab}}|\le \frac{b'(n/2)}{\sqrt{\log(n/2)}}+2\frac{(n/2)}{\sqrt{\log(n/2)}}.
\end{equation}
Now, the right hand side of~\eqref{aux1} is less than $2n/\sqrt{\log n}$ for each $n\ge 20\, 604$.
The computation when $r\in \{3,4\}$ is analogous using $a(R)\le 1+\log(3)$ when $r=3$, and $a(R)\le 3+\log(3)$ when $r=4$.\\

\end{document}